\documentclass[a4paper]{amsproc}
\usepackage{amssymb}
\usepackage{amscd} 

\usepackage{tikz}
\usetikzlibrary{positioning}
\usepackage{subcaption}
\usepackage{graphicx}

\newtheorem{df}{Definition}[section]

\newtheorem{thm}{Theorem}[section]

\newtheorem{cor}{Corollary}[section]
\newtheorem{lm}{Lemma}[section]
\theoremstyle{plain}

\theoremstyle{definition}
 \newtheorem{exm}{Example}[section]

\numberwithin{equation}{section}

\renewcommand{\ge}{\geqslant}

\title[Topological indices of k-th subdivision and semi total point graphs]{TOPOLOGICAL INDICES OF K-TH SUBDIVISION AND SEMI TOTAL POINT GRAPHS}



\author[De]{\bfseries Nilanjan De}

\address{
Department of Basic Sciences and Humanities \\ 
Calcutta Institute of Engineering and Management \\ 
Kolkata\\
India\\}
\email{de.nilanjan@rediffmail.com}

\begin{document}

\begin{abstract}
Graph theory has provided a very useful tool, called topological indices which are a number obtained from the graph $G$ with the property that every graph $H$ isomorphic to $G$, value of a topological index must be same for both $G$ and $H$. In this article, we present exact expressions for some topological indices of k-th subdivision graph and semi total point graphs respectively, which are a generalization of ordinary subdivision and semi total graph for $k\ge 1$.
\end{abstract}

\maketitle

\section{Introduction}

Let, $G=(V,E)$ be a connected, undirected simple graph with vertex set $V=V(G)$ and edge set $E=E(G)$. The degree of a vertex $v$ in $G$ is defined as the number of edges incident to $v$ and denoted by ${{d}_{G}}(v)$. Let $n$ and $m$ denote the order and size of the graph $G$. In this paper we consider only connected graphs without any self loop and parallel edges.
In chemical graph theory, a topological index is a number obtained from a graph which is structurally invariant, that is, every graph $H$ isomorphic to $G$, value of a topological index must be same for both $G$ and $H$. Among different topological indices the vertex degree based topological indices has shown there applicability in chemistry, biochemistry, nanotechnology.

The most popular Zagreb indices were introduced more than forty years ago \cite{gutm72} by Gutman and Trinajesti\'{c} where they have examined the dependence of the total $\pi$-electron energy on molecular structure. These indices are denoted by $M_1(G)$ and $M_2(G)$ and respectively defined as\\
\[{{M}_{1}}(G)=\sum\limits_{v\in V(G)}{{{d}_{G}}{{(v)}^{2}}}=\sum\limits_{uv\in E(G)}{[{{d}_{G}}(u)+{{d}_{G}}(v)]}\] and \[{{M}_{2}}(G)=\sum\limits_{uv\in E(G)}{{{d}_{G}}(u){{d}_{G}}(v)}.\]
These indices have been studied intensively by researchers in recent past (see \cite{fath11,zho05,das15,gut15,basa15}).\\

The ``forgotten topological index" or F-index of a graph $G$ is denoted by $F(G)$ and was introduced in the same paper where first and second Zagreb indices were introduced \cite{gutm72} and is defined as

\[F(G)=\sum\limits_{v\in V(G)}{{{d}_{G}}{{(v)}^{3}}}=\sum\limits_{uv\in E(G)}{[{{d}_{G}}{{(u)}^{2}}+{{d}_{G}}{{(v)}^{2}}]}.\]
We encourage the interested readers to consult the papers \cite{furt15,nde17,de15,de18,den17} for some recent study of this index.

The multiplicative variants of additive graph invariants was introduced by Todeschini et al. \cite{tod10,tod10a}, which applied to the Zagreb indices would lead to the first and second Multiplicative Zagreb Indices. Thus the multiplicative Zagreb indices are defined as
\[{{\Pi }_{1}}(G)=\prod\limits_{v\in V(G)}{{{d}_{G}}{{(v)}^{2}}}\]   and   \[{{\Pi }_{2}}(G)=\prod\limits_{uv\in E(G)}{{{d}_{G}}(u){{d}_{G}}(v)}.\]
For different studies on these two indices see \cite{eli12,gut11,xu12,liu12,reti12,basa16}.

One modified version of Zagreb index, named as hyper Zagreb index, was introduced by Shirrdel et al. in \cite{shir13} and is defined as
\[{HM}(G)=\sum\limits_{uv\in E(G)}{[{{d}_{G}}(u)+{{d}_{G}}(v)]^2}.\]
We refer the reader to \cite{ndeh17,gao17}, for some recent study and application of hyper Zagreb index.

The Symmetric division deg index of a graph $G$ is defined as \cite{vuk10}
\[SSD(G)=\sum\limits_{uv\in E(G)}{[\frac{\max \{{{d}_{G}}(u),{{d}_{G}}(v)\}}{\min \{{{d}_{G}}(u),{{d}_{G}}(v)\}}+\frac{\min \{{{d}_{G}}(u),{{d}_{G}}(v)\}}{\max \{{{d}_{G}}(u),{{d}_{G}}(v)\}}]}.\]
For different recent study of this index see \cite{alex14,gupta16}.

There are various studies of different derived graphs in recent literature. The subdivision related derived graphs were introduced by Sampathkumar  in \cite{sam1,sam2,sam3} and named as semitotal-point graph and semitotal-line graph. First of all the subdivision graph $S(G)$ of a graph $G$ is defined as follows:

\begin{df}
The subdivision graph $S(G)$ is obtained from $G$ by adding a new vertex corresponding to every edge of $G$, that is, each edge of $G$ is replaced by a path of length two.
\end{df}

Hande et. al  in \cite{han13}, defined a generalization of $S(G)$, called the k-th subdivision graph, denoted by ${{S}_{k}}(G)$ and is defined as follows:

\begin{df}
The k-th subdivision graph of $G$, denoted by ${{S}_{k}}(G)$, is the graph obtained by inserting $k$ number of new vertices to each edge of $G$. Thus the graph of ${{S}_{k}}(G)$ consists of total $(n+km)$ number of vertices and has  $(1+k)m$ number of edges.
\end{df}
The semi total point graph $R(G)$ of a graph $G$ is defined as follows:
\begin{df}
The semi total point graph $R(G)$ is obtained from $G$ by adding a new vertex corresponding to every edge of $G$,  then joining each new vertex to the end vertices of the corresponding edge that is, each edge of $G$ is replaced by a triangle.
\end{df}

Jog et. al in \cite{jog12}, introduced a generalization of $R(G)$, denoted by ${{R}_{k}}(G)$, called the k-th semi total point graph of $G$. The k-th semi total point graph of $G$ is defined as follows:

\begin{df}
The k-th semi total point graph ${{R}_{k}}(G)$ is the graph obtained from G by adding k vertices corresponding to each edge and connecting them to the end points of the edge considered. Clearly the graph of ${{R}_{k}}(G)$ is of order (n+mk) and has (1+2k)m number of edges.
\end{df}

In \cite{meh14}, Mehranian obtained formulas for the Laplacian polynomial and Kirchhoff index of the k-th semi total point graphs. For different study of different derived graphs see the recent papers \cite{de17,sara17,pal16}.
In this article, we determine the different topological indices such as first and second Zagreb indices, multiplicative Zagreb indices, F-index, hyper-Zagreb index and Symmetric division deg index of k-th subdivision and semi total point graphs respectively.

\section{Main Results}

In this section, first we derive exact expressions of certain topological indices of k-th subdivision graph and then proceeding similarly we obtain different topological indices of k-th semi total point graph respectively.

\subsection{K-th subdivision graphs}

The k-th subdivision graph of $G$ is the graph obtained by inserting $k$ number of new vertices of degree 2 to each edge of $G$. Thus in ${{S}_{k}}(G)$ there are km number vertices of degree 2 and the remaining vertices are of same as $G$. The degree of the vertices of  ${{S}_{k}}(G)$ are given as follows
\[{{d}_{{{S}_{k}}(G)}}(v)=\left\{ \begin{array}{ll}
  {{d}_{G}}(v),\,if\,v\in V(G) \\[2mm]
  2,\,if\,v\in V({{S}_{k}}(G))\backslash V(G). \\[2mm]
\end{array} \right.\]

From definition, it is clear that ${{S}_{0}}(G)=G$ and ${{S}_{1}}(G)=S(G)$. In the following theorems, we now calculate different topological indices k-th subdivision graph ${{S}_{k}}(G)$ respectively.

\begin{thm}
Let $G$ be a connected graphs. Then
\[{{M}_{1}}({{S}_{k}}(G))={{M}_{1}}(G)+4km.\]
\end{thm}
\begin{proof}
Since ${{S}_{k}}(G)$ has $(n+mk)$ number of vertices and $(1+k)m$ number of edges, from definition of first Zagreb index, we have\\
\begin{eqnarray*}
{{M}_{1}}({{S}_{k}}(G))&=&\sum\limits_{v\in V({{S}_{k}}(G))}{{{d}_{{{S}_{k}}(G)}}{{(v)}^{2}}}\\
   &=&\sum\limits_{v\in V({{S}_{k}}(G))\cap V(G)}{{{d}_{{{S}_{k}}(G)}}{{(v)}^{2}}}+\sum\limits_{v\in V({{S}_{k}}(G))\backslash V(G)}{{{d}_{{{S}_{k}}(G)}}{{(v)}^{2}}}\\
\end{eqnarray*}
Also, since for ${{S}_{k}}(G)$, there are $n$ number of vertices are of degree ${{d}_{G}}(v)$ and $mk$ number of vertices are of degree 2, we have\\
\begin{eqnarray*}
{{M}_{1}}({{S}_{k}}(G))&=&\sum\limits_{v\in V(G)}{{{d}_{G}}{{(v)}^{2}}}+\sum\limits_{v\in V({{S}_{k}}(G))\backslash V(G)}{{{2}^{2}}}\\
 &=&{{M}_{1}}(G)+4km.
 \end{eqnarray*}
Hence the desired result follows.
\end{proof}

\begin{thm}
Let $G$ be a connected graphs. Then
\[{{M}_{2}}({{S}_{k}}(G))=2{{M}_{1}}(G)+4(k-1)m.\]
\end{thm}

\begin{proof}
We know that, in ${{S}_{k}}(G)$, there are $n$ number of vertices are of degree ${{d}_{G}}(v)$ and $mk$ number of vertices are of degree 2.  Also a vertex $v$ in $G$ is adjacent with ${{d}_{G}}(v)$ number vertices of degree 2 in ${{S}_{k}}(G)$ and for every edge in $G$ there must be (k-1) edges in ${{S}_{k}}(G)$ with both the end vertices are of degree 2.  So from definition of  second Zagreb  index of a graph we have\\
\begin{eqnarray*}
{{M}_{2}}({{S}_{k}}(G))&=&\sum\limits_{uv\in E({{S}_{k}}(G))}{{{d}_{{{S}_{k}}(G)}}}(u){{d}_{{{S}_{k}}(G)}}(v)\\\\
                 &=&\sum\limits_{\begin{smallmatrix}
 uv\in E({{S}_{k}}(G)) \\
 u\in V(G),v\in V({{S}_{k}}(G))\backslash V(G)
\end{smallmatrix}}{{{d}_{{{S}_{k}}(G)}}(u){{d}_{{{S}_{k}}(G)}}(v)}\\
&&+\sum\limits_{\begin{smallmatrix}
 uv\in E({{S}_{k}}(G)) \\
 u,v\in V({{S}_{k}}(G))\backslash V(G)
\end{smallmatrix}}{{{d}_{{{S}_{k}}(G)}}(u){{d}_{{{S}_{k}}(G)}}(v)}\\\\
                &=&\sum\limits_{\begin{smallmatrix}
 uv\in E({{S}_{k}}(G)) \\
 u\in V(G),v\in V({{S}_{k}}(G))\backslash V(G)
\end{smallmatrix}}{2{{d}_{G}}(u)}+\sum\limits_{\begin{smallmatrix}
 uv\in E({{S}_{k}}(G)) \\
 u,v\in V({{S}_{k}}(G)) \backslash V(G)
\end{smallmatrix}}{2.2}\\\\
                &=&2{{M}_{1}}(G)+4(k-1)m.
\end{eqnarray*}
Which is the desired result.
\end{proof}

\begin{thm}
 Let $G$ be a connected graphs. Then \\
\[F({{S}_{k}}(G))=F(G)+8km.\]
\end{thm}
\begin{proof}
As ${{S}_{k}}(G)$ has $(n+mk)$ number of vertices and $(1+k)m$ number of edges, from definition of F-index index, we have
\begin{eqnarray*}
F({{S}_{k}}(G))&=&\sum\limits_{v\in V({{S}_{k}}(G))}{{{d}_{{{S}_{k}}(G)}}{{(v)}^{3}}}\\
&=&\sum\limits_{v\in V({{S}_{k}}(G))\cap V(G)}{{{d}_{{{S}_{k}}(G)}}{{(v)}^{3}}}+\sum\limits_{v\in V({{S}_{k}}(G))\backslash V(G)}{{{d}_{{{S}_{k}}(G)}}{{(v)}^{3}}}.
\end{eqnarray*}
Again, we have for ${{S}_{k}}(G)$, there are $n$ number of vertices are of degree ${{d}_{G}}(v)$ and $mk$  number of vertices are of degree 2, we have similarly
\begin{eqnarray*}
F({{S}_{k}}(G))&=&\sum\limits_{v\in V(G)}{{{d}_{G}}{{(v)}^{3}}}+\sum\limits_{v\in V({{S}_{k}}(G))\backslash V(G)}{{{2}^{3}}}\\
&=&F(G)+8km.
\end{eqnarray*}
Hence we get the desired result.
\end{proof}

\begin{thm}
 Let $G$ be a connected graphs. Then
\[{{\Pi }_{1}}({{S}_{k}}(G))={{4}^{km}}{{\Pi }_{1}}(G).\]
\end{thm}
\begin{proof}
As ${{S}_{k}}(G)$ has $(n+mk)$ number of vertices and $(1+k)m$ number of edges, from definition of first multiplicative Zagreb index index, we have
\begin{eqnarray*}
{{\Pi }_{1}}({{S}_{k}}(G))&=&\prod\limits_{v\in V({{S}_{k}}(G))}{{{d}_{{{S}_{k}}(G)}}{{(v)}^{2}}}\\
   &=&\prod\limits_{v\in V({{S}_{k}}(G))\cap V(G)}{{{d}_{{{S}_{k}}(G)}}{{(v)}^{2}}}\prod\limits_{v\in V({{S}_{k}}(G))\backslash V(G)}{{{d}_{{{S}_{k}}(G)}}{{(v)}^{2}}}.
\end{eqnarray*}
Again, for the graph ${{S}_{k}}(G)$, there are $n$ number of vertices are of degree ${{d}_{G}}(v)$ and $mk$ number of vertices are of degree 2 and hence we can write
 \begin{eqnarray*}
 {{\Pi }_{1}}({{S}_{k}}(G))&=&\prod\limits_{v\in V(G)}{{{d}_{G}}{{(v)}^{2}}}\prod\limits_{v\in V({{S}_{k}}(G))\backslash V(G)}{{{2}^{2}}}\\
   &=&{{4}^{km}}{{\Pi }_{1}}(G).
   \end{eqnarray*}
Which is the desired result.
\end{proof}

\begin{thm}
Let $G$ be a connected graphs. Then
 \[{{\Pi }_{2}}({{S}_{k}}(G))={{4}^{km}}{{\Pi }_{2}}(G).\]
\end{thm}
\begin{proof}
We know that, ${{S}_{k}}(G)$ has $(n+mk)$ number of vertices and $(1+k)m$ number of edges. Thus, from definition of second multiplicative Zagreb index, we have
\begin{eqnarray*}
{{\Pi }_{2}}({{S}_{k}}(G))&=&\prod\limits_{uv\in E({{S}_{k}}(G))}{{{d}_{{{S}_{k}}(G)}}(u){{d}_{{{S}_{k}}(G)}}(v)}\\
&=&\prod\limits_{v\in V({{S}_{k}}(G))\cap V(G)}{{{d}_{{{S}_{k}}(G)}}{{(v)}^{{{d}_{{{S}_{k}}(G)}}(v)}}}\prod\limits_{v\in V({{S}_{k}}(G))\backslash V(G)}{{{d}_{{{S}_{k}}(G)}}{{(v)}^{{{d}_{{{S}_{k}}(G)}}(v)}}}.
\end{eqnarray*}
Also, we have for ${{S}_{k}}(G)$, there are $n$ number of vertices are of degree ${{d}_{G}}(v)$ and $mk$ number of vertices are of degree 2. Then, we have
\begin{eqnarray*}
{{\Pi }_{2}}({{S}_{k}}(G))&=&\prod\limits_{v\in V(G)}{{{d}_{G}}{{(v)}^{{{d}_{G}}(v)}}}\prod\limits_{v\in V({{S}_{k}}(G))\backslash V(G)}{{{2}^{2}}}\\
 &=&{{4}^{km}}{{\Pi }_{2}}(G).
\end{eqnarray*}
Hence the desired result follows.
\end{proof}

\begin{thm}
Let $G$ be a connected graphs. Then
\[HM({{S}_{k}}(G))=F(G)+4{{M}_{1}}(G)+16km-8m.\]
\end{thm}
\begin{proof}
Since for every edge in $G$ there are $(k-1)$ edges in ${{S}_{k}}(G)$ with both the end vertices are of degree 2 and also a vertex $v$ in $G$ is adjacent with ${{d}_{G}}(v)$ number vertices of degree 2.  So from definition of hyper Zagreb index of a graph, we have
\begin{eqnarray*}
HM({{S}_{k}}(G))&=&\sum\limits_{uv\in E({{S}_{k}}(G))}{{{[{{d}_{{{S}_{k}}(G)}}(u)+{{d}_{{{S}_{k}}(G)}}(v)]}^{2}}}\\
       &=&\sum\limits_{\begin{smallmatrix}
 uv\in E({{S}_{k}}(G)) \\
 u\in V(G),v\in V({{S}_{k}}(G)) \backslash V(G)
\end{smallmatrix}}{{{[{{d}_{{{S}_{k}}(G)}}(u)+{{d}_{{{S}_{k}}(G)}}(v)]}^{2}}}\\
&&+\sum\limits_{\begin{smallmatrix}
 uv\in E({{S}_{k}}(G)) \\
 u,v\in V({{S}_{k}}(G))\backslash V(G)
\end{smallmatrix}}{{{[{{d}_{{{S}_{k}}(G)}}(u)+{{d}_{{{S}_{k}}(G)}}(v)]}^{2}}}\\
                      &=&\sum\limits_{u\in V(G)}{{{[2+{{d}_{G}}(u)]}^{2}}{{d}_{G}}(u)}+\sum\limits_{\begin{smallmatrix}
 uv\in E({{S}_{k}}(G)) \\
 u,v\in V({{S}_{k}}(G))\backslash V(G)
\end{smallmatrix}}{{{[2+2]}^{2}}}\\
                     &=&\sum\limits_{u\in V(G)}{[{{d}_{G}}{{(u)}^{3}}+4{{d}_{G}}{{(u)}^{2}}+4{{d}_{G}}(u)]}+16(k-1)m\\
 &=&F(G)+4{{M}_{1}}(G)+8m+16km-16m.
\end{eqnarray*}
Which is the required expression.
\end{proof}

\begin{thm}
Let $G$ be a connected graphs. Then
\[SSD({{S}_{k}}(G))=\frac{1}{2}{{M}_{1}}(G)+2(k-1)m+2n.\]
\end{thm}
\begin{proof}

Since for ${{S}_{k}}(G)$, there are $n$ number of vertices are of degree ${{d}_{G}}(v)$ and $mk$ number of vertices are of degree 2 and for every edge in $G$ there are $(k-1)$ edges in ${{S}_{k}}(G)$ with both the end vertices are of degree 2. Also a vertex $v$ in $G$ is adjacent with ${{d}_{G}}(v)$ number vertices of degree 2. So from definition of Symmetric division deg index of a graph we have
\begin{eqnarray*}
SSD({{S}_{k}}(G))&=&\sum\limits_{uv\in E({{S}_{k}}(G))}{\frac{{{d}_{{{S}_{k}}(G)}}{{(u)}^{2}}+{{d}_{{{S}_{k}}(G)}}{{(v)}^{2}}}{{{d}_{{{S}_{k}}(G)}}(u){{d}_{{{S}_{k}}(G)}}(v)}}\\
                     &=&\sum\limits_{\begin{smallmatrix}
 uv\in E({{S}_{k}}(G)) \\
 u\in V(G),v\in V({{S}_{k}}(G))\backslash V(G)
\end{smallmatrix}}{\frac{{{d}_{{{S}_{k}}(G)}}{{(u)}^{2}}+{{d}_{{{S}_{k}}(G)}}{{(v)}^{2}}}{{{d}_{{{S}_{k}}(G)}}(u){{d}_{{{S}_{k}}(G)}}(v)}}\\
&&+\sum\limits_{\begin{smallmatrix}
 uv\in E({{S}_{k}}(G)) \\
 u,v\in V({{S}_{k}}(G))\backslash V(G)
\end{smallmatrix}}{\frac{{{d}_{{{S}_{k}}(G)}}{{(u)}^{2}}+{{d}_{{{S}_{k}}(G)}}{{(v)}^{2}}}{{{d}_{{{S}_{k}}(G)}}(u){{d}_{{{S}_{k}}(G)}}(v)}}\\
                     &=&\sum\limits_{u\in V(G)}{\frac{{{2}^{2}}+{{d}_{G}}{{(u)}^{2}}}{2{{d}_{G}}(u)}{{d}_{G}}(u)}+\sum\limits_{\begin{smallmatrix}
                      uv\in E({{S}_{k}}(G)) \\
 u,v\in V({{S}_{k}}(G))\backslash V(G)
\end{smallmatrix}}{\frac{{{2}^{2}}+{{2}^{2}}}{2.2}}\\
                    &=&\frac{1}{2}\sum\limits_{u\in V(G)}{[{{d}_{G}}{{(u)}^{2}}+4]}+2(k-1)m\\
                   &=&\frac{1}{2}{{M}_{1}}(G)+2n+2(k-1)m.
 \end{eqnarray*}
Thus the desired result follows.
\end{proof}
\begin{cor}
From the above theorems the following results follows as a direct consequence.\\
(i) ${{M}_{1}}({{S}_{k}}(G))={{M}_{1}}({{S}_{k-1}}(G))+4m,$\\
(ii) ${{M}_{2}}({{S}_{k}}(G))={{M}_{2}}({{S}_{k-1}}(G))+4m,$\\
(iii) $F({{S}_{k}}(G))=F({{S}_{k-1}}(G))+8m,$\\
(iv) ${{\Pi }_{1}}({{S}_{k}}(G))={{4}^{m}}{{\Pi }_{1}}({{S}_{k-1}}(G)),$\\
(v) ${{\Pi }_{2}}({{S}_{k}}(G))={{4}^{m}}{{\Pi }_{2}}({{S}_{k-1}}(G)),$\\
(vi) $HM({{S}_{k}}(G))=HM({{S}_{k-1}}(G))+16m,$\\
(vii) $SSD({{S}_{k}}(G))=SSD({{S}_{k-1}}(G))+2m.$
\end{cor}

\begin{cor}
Let G be a r-regular graph with $n$ vertices and $m\ (=\frac{nr}{2})$ number of edges. Then using theorem 2.1-2.7, the following result follows after direct calculation:\\
(i) ${{M}_{1}}({{S}_{k}}(G))=n{{r}^{2}}+2nkr,$\\
(ii) ${{M}_{2}}({{S}_{k}}(G))=2nr(r+k-1),$\\
(iii) $F({{S}_{k}}(G))=n{{r}^{3}}+4nkr,$\\
(iv) ${{\Pi }_{1}}({{S}_{k}}(G))={{2}^{knr}}{{r}^{2n}},$\\
(i) ${{\Pi }_{2}}({{S}_{k}}(G))={{2}^{nkr}}{{r}^{nr}},$\\
(i) $HM({{S}_{k}}(G))=nr({{r}^{2}}+4r+8k-4),$\\
(i) $SSD({{S}_{k}}(G))=\frac{1}{2}n{{r}^{2}}+nr(k-1)+2n.$\\\\
\end{cor}

Note that, from the results obtained in theorem 2.1-2.7, we can easily obtain the result for ordinary subdivision graph by putting $k=1$. For instance, putting $k=1$ in theorem 2.1-2.7 we get the same results of subdivision graph S(G) for first and second Zagreb index, F-index and multiplicative Zagreb indies as in \cite{gut15,basa15,de18,basa16} respectively. Thus here we generalized the above results for some $k\ge 1$.

\subsection{K-th semi total point graphs}

The k-th semi total point graph of $G$, denoted by ${{R}_{k}}(G)$, is the graph obtained by adding $k$ number of vertices to each edge of $G$ and joining them to the end points of the respective edges. Clearly the graph of ${{R}_{k}}(G)$ is of order $(n+mk)$ and has $(1+2k)m$ number of edges. The degree of the vertices of  ${{R}_{k}}(G)$ are given as follows

\[{{d}_{{{R}_{k}}(G)}}(v)=\left\{ \begin{array}{ll}
  (k+1){{d}_{G}}(v),\,if\,v\in V(G) \\[2mm]
  2,\,if\,v\in V({{R}_{k}}(G))\backslash V(G).\\[2mm]
  \end{array} \right.\]
Also, it is clear that ${{R}_{0}}(G)=G$ and ${{R}_{1}}(G)=R(G)$. In the following theorems, we now calculate different topological indices k-th semi total graph ${{R}_{k}}(G)$ respectively.

\begin{thm}
Let $G$ be a connected graphs. Then
\[{{M}_{1}}({{R}_{k}}(G))={{(k+1)}^{2}}{{M}_{1}}(G)+4km.\]	
\end{thm}
\begin{proof}
Since, ${{R}_{k}}(G)$ has $(n+mk)$ number of vertices and $(1+2k)m$ number of edges, from definition of first Zagreb index, we have
\begin{eqnarray*}
{{M}_{1}}({{R}_{k}}(G))&=&\sum\limits_{v\in V({{R}_{k}}(G))}{{{d}_{{{R}_{k}}(G)}}{{(v)}^{2}}}\\
 &=&\sum\limits_{v\in V({{R}_{k}}(G))\cap V(G)}{{{d}_{{{R}_{k}}(G)}}{{(v)}^{2}}}+\sum\limits_{v\in V({{R}_{k}}(G))\backslash V(G)}{{{d}_{{{R}_{k}}(G)}}{{(v)}^{2}}}.
\end{eqnarray*}
Again, since in ${{R}_{k}}(G)$ there are $n$ number of vertices are of degree $(k+1){{d}_{G}}(v)$ and $mk$ number of vertices are of degree 2, we have
\begin{eqnarray*}
{{M}_{1}}({{R}_{k}}(G))&=&\sum\limits_{v\in V(G)}{{{(k+1)}^{2}}{{d}_{G}}{{(v)}^{2}}}+\sum\limits_{v\in V({{R}_{k}}(G))\backslash V(G)}{{{2}^{2}}}\\
  &=&{{(k+1)}^{2}}{{M}_{1}}(G)+4km,
\end{eqnarray*}
which proves the desired result.
\end{proof}

%

\begin{thm}
Let $G$ be a connected graphs. Then
\[{{M}_{2}}({{R}_{k}}(G))=2k(k+1){{M}_{1}}(G)+{{(k+1)}^{2}}{{M}_{2}}(G).\]
\end{thm}
\begin{proof}
From definition of second Zagreb index, we have
\begin{eqnarray*}
{{M}_{2}}({{R}_{k}}(G))&=&\sum\limits_{uv\in E({{R}_{k}}(G))}{{{d}_{{{R}_{k}}(G)}}(u){{d}_{{{R}_{k}}(G)}}(v)}\\
                     &=&\sum\limits_{\begin{smallmatrix}
 uv\in E({{S}_{k}}(G)) \\
 u\in V(G),v\in V({{S}_{k}}(G))\backslash V(G)
\end{smallmatrix}}{{{d}_{{{R}_{k}}(G)}}(u){{d}_{{{R}_{k}}(G)}}(v)}\\
&&+\sum\limits_{\begin{smallmatrix}
 uv\in E({{S}_{k}}(G)) \\
 u,v\in V({{S}_{k}}(G)) \backslash V(G)
\end{smallmatrix}}{{{d}_{{{R}_{k}}(G)}}(u){{d}_{{{R}_{k}}(G)}}(v)}\\
                    &=&\sum\limits_{\begin{smallmatrix}
 uv\in E({{R}_{k}}(G)) \\
 u\in V(G),v\in V({{R}_{k}}(G))\backslash V(G)
\end{smallmatrix}}{2(k+1){{d}_{G}}(u)}+\sum\limits_{\begin{smallmatrix}
 uv\in E({{R}_{k}}(G)) \\
 u,v\in V(G)
\end{smallmatrix}}{{{(k+1)}^{2}}{{d}_{G}}(u){{d}_{G}}(v)}\\
                    &=&2k(k+1){{M}_{1}}(G)+{{(k+1)}^{2}}{{M}_{2}}(G).\\
\end{eqnarray*}
Hence, we get the required result.
\end{proof}

\begin{thm}
Let $G$ be a connected graphs. Then
\[F({{R}_{k}}(G))={{(k+1)}^{3}}F(G)+8km.\]
\end{thm}
\begin{proof}
From definition of F-index and since ${{R}_{k}}(G)$ has $(n+mk)$ number of vertices and $(1+2k)m$ number of edges, we have
\begin{eqnarray*}
F({{R}_{k}}(G))&=&\sum\limits_{v\in V({{R}_{k}}(G))}{{{d}_{{{R}_{k}}(G)}}{{(v)}^{3}}}\\
&=&\sum\limits_{v\in V({{R}_{k}}(G))\cap V(G)}{{{d}_{{{R}_{k}}(G)}}{{(v)}^{3}}}+\sum\limits_{v\in V({{R}_{k}}(G))\backslash V(G)}{{{d}_{{{R}_{k}}(G)}}{{(v)}^{3}}}.\\
\end{eqnarray*}
Also, in ${{R}_{k}}(G)$ there are $n$ number of vertices are of degree $(k+1){{d}_{G}}(v)$ and the remaining $mk$ number of vertices are of degree 2, so we have
\begin{eqnarray*}
  F({{R}_{k}}(G))&=&\sum\limits_{v\in V(G)}{{{(k+1)}^{3}}{{d}_{G}}{{(v)}^{3}}}+\sum\limits_{v\in V({{R}_{k}}(G))\backslash V(G)}{{{2}^{3}}}\\
   &=&{{(k+1)}^{3}}F(G)+8km.
 \end{eqnarray*}
Which is the desired expression.
\end{proof}
\begin{thm}
 Let $G$ be a connected graphs. Then
\[{{\Pi }_{1}}({{R}_{k}}(G))={{4}^{km}}{{(k+1)}^{2n}}{{\Pi }_{1}}(G).\]
\end{thm}
\begin{proof}
Since, ${{R}_{k}}(G)$ has $(n+mk)$ number of vertices and $(1+2k)m$ number of edges, from definition of first multiplicative Zagreb index, we have
\begin{eqnarray*}
{{\Pi }_{1}}({{R}_{k}}(G))&=&\prod\limits_{v\in V({{R}_{k}}(G))}{{{d}_{{{R}_{k}}(G)}}{{(v)}^{2}}}\\
  &=&\prod\limits_{v\in V({{R}_{k}}(G))\cap V(G)}{{{d}_{{{R}_{k}}(G)}}{{(v)}^{2}}}\prod\limits_{v\in V({{R}_{k}}(G))\backslash V(G)}{{{d}_{{{R}_{k}}(G)}}{{(v)}^{2}}}.\\
  \end{eqnarray*}
Again, similarly the graph ${{R}_{k}}(G)$ has $n$ number of vertices of degree $(k+1){{d}_{G}}(v)$ and $mk$ number of vertices are of degree 2, we have
\begin{eqnarray*}
{{\Pi }_{1}}({{R}_{k}}(G))&=&\prod\limits_{v\in V(G)}{{{(k+1)}^{2}}{{d}_{G}}{{(v)}^{2}}}\prod\limits_{v\in V({{R}_{k}}(G))\backslash V(G)}{{{2}^{2}}}\\
 &=&{{4}^{km}}{{(k+1)}^{2n}}{{\Pi }_{1}}(G).
\end{eqnarray*}
Hence the desired result follows.
\end{proof}
\begin{thm}
Let $G$ be a connected graphs. Then
\[{{\Pi }_{2}}({{R}_{k}}(G))={{4}^{km}}{{\left\{ {{(k+1)}^{2m}}{{\Pi }_{1}}(G) \right\}}^{(k+1)}}.\]
\end{thm}
\begin{proof}
We have from definition of second multiplicative Zagreb index
\begin{eqnarray*}
{{\Pi }_{2}}({{R}_{k}}(G))&=&\prod\limits_{uv\in E({{R}_{k}}(G))}{{{d}_{{{R}_{k}}(G)}}(u){{d}_{{{R}_{k}}(G)}}(v)}\\
&=&\prod\limits_{v\in V({{R}_{k}}(G))\cap V(G)}{{{d}_{{{R}_{k}}(G)}}{{(v)}^{{{d}_{{{R}_{k}}(G)}}(v)}}}\prod\limits_{v\in V({{R}_{k}}(G))\backslash V(G)}{{{d}_{{{R}_{k}}(G)}}{{(v)}^{{{d}_{{{R}_{k}}(G)}}(v)}}}.
\end{eqnarray*}
Again, since ${{R}_{k}}(G)$ has
$(n+mk)$ number of vertices and out of which $n$ numbers of vertices are of degree $(k+1){{d}_{G}}(v)$ and $mk$ number of vertices are of degree 2, we have
\begin{eqnarray*}
{{\Pi }_{2}}({{R}_{k}}(G))&=&\prod\limits_{v\in V(G)}{{{\{(k+1){{d}_{G}}(v)\}}^{(k+1){{d}_{G}}(v)}}}\prod\limits_{v\in V({{R}_{k}}(G))\backslash V(G)}{{{2}^{2}}}\\
                &=&{{(k+1)}^{2m(k+1)}}{{\left\{ \prod\limits_{v\in V(G)}{{{d}_{G}}{{(v)}^{{{d}_{G}}(v)}}} \right\}}^{(k+1)}}\prod\limits_{v\in V({{R}_{k}}(G))\backslash V(G)}{{{2}^{2}}}\\
               &=&{{4}^{km}}{{\left\{ {{(k+1)}^{2m}}{{\Pi }_{1}}(G) \right\}}^{(k+1)}}.
\end{eqnarray*}
Which proves the desired expression.
\end{proof}

\begin{thm}
 Let $G$ be a connected graphs. Then
\[HM({{R}_{k}}(G))={{(k+1)}^{2}}HM(G)+k{{(k+1)}^{2}}F(G)+4k(k+1){{M}_{1}}(G)+8km.\]
\end{thm}
\begin{proof}
From definition of hyper-Zagreb index, we have
\begin{eqnarray*}
HM({{R}_{k}}(G))&=&\sum\limits_{uv\in E({{R}_{k}}(G))}{{{[{{d}_{{{R}_{k}}(G)}}(u)+{{d}_{{{R}_{k}}(G)}}(v)]}^{2}}}\\
       &=&\sum\limits_{\begin{smallmatrix}
 uv\in E({{R}_{k}}(G)) \\
 u\in V(G),v\in V({{R}_{k}}(G))\backslash V(G)
\end{smallmatrix}}{{{[{{d}_{{{R}_{k}}(G)}}(u)+{{d}_{{{R}_{k}}(G)}}(v)]}^{2}}}\\
&&+\sum\limits_{\begin{smallmatrix}
 uv\in E({{R}_{k}}(G)) \\
 u,v\in V(G)
\end{smallmatrix}}{{{[{{d}_{{{R}_{k}}(G)}}(u)+{{d}_{{{R}_{k}}(G)}}(v)]}^{2}}}\\
                      &=&\sum\limits_{u\in V(G)}{{{[2+(k+1){{d}_{G}}(u)]}^{2}}k{{d}_{G}}(u)}\\
                      &&+\sum\limits_{uv\in E(G)}{{{[(k+1){{d}_{G}}(u)+(k+1){{d}_{G}}(v)]}^{2}}}\\
                     &=&k\sum\limits_{u\in V(G)}{[{{(k+1)}^{2}}{{d}_{G}}{{(u)}^{3}}+4(k+1){{d}_{G}}{{(u)}^{2}}+4{{d}_{G}}(u)]}\\
                     &&+{{(k+1)}^{2}}HM(G)\\
                     &=&k{{(k+1)}^{2}}F(G)+4k(k+1){{M}_{1}}(G)+8km+{{(k+1)}^{2}}HM(G).
 \end{eqnarray*}
Thus the required result follows.
\end{proof}
\begin{thm}
Let $G$ be a connected graphs. Then
\[SSD({{R}_{k}}(G))=SSD(G)+\frac{1}{2}(k+1){{M}_{1}}(G)+\frac{2n}{(k+1)}.\]
\end{thm}
\begin{proof}
From definition of Symmetric division deg index and k-th semi total graph, we have
\begin{eqnarray*}
SSD({{R}_{k}}(G))&=&\sum\limits_{uv\in E({{R}_{k}}(G))}{\frac{{{d}_{{{R}_{k}}(G)}}{{(u)}^{2}}+{{d}_{{{R}_{k}}(G)}}{{(v)}^{2}}}{{{d}_{{{R}_{k}}(G)}}(u){{d}_{{{R}_{k}}(G)}}(v)}}\\
                      &=&\sum\limits_{\begin{smallmatrix}
 uv\in E({{R}_{k}}(G)) \\
 u\in V(G),v\in V({{R}_{k}}(G))\backslash V(G)
\end{smallmatrix}}{\frac{{{d}_{{{R}_{k}}(G)}}{{(u)}^{2}}+{{d}_{{{R}_{k}}(G)}}{{(v)}^{2}}}{{{d}_{{{R}_{k}}(G)}}(u){{d}_{{{R}_{k}}(G)}}(v)}}\\
\end{eqnarray*}
\begin{eqnarray*}
&&+\sum\limits_{\begin{smallmatrix}
 uv\in E({{R}_{k}}(G)) \\
 u,v\in V({{R}_{k}}(G)) \backslash V(G)
\end{smallmatrix}}{\frac{{{d}_{{{R}_{k}}(G)}}{{(u)}^{2}}+{{d}_{{{R}_{k}}(G)}}{{(v)}^{2}}}{{{d}_{{{R}_{k}}(G)}}(u){{d}_{{{R}_{k}}(G)}}(v)}}\\
                     &=&\sum\limits_{u\in V(G)}{\frac{{{2}^{2}}+{{(k+1)}^{2}}{{d}_{G}}{{(u)}^{2}}}{2(k+1){{d}_{G}}(u)}{{d}_{G}}(u)}\\
                     &&+\sum\limits_{\begin{smallmatrix}
 uv\in E({{S}_{k}}(G)) \\
 u,v\in V({{S}_{k}}(G))\backslash V(G)
\end{smallmatrix}}{\frac{{{(k+1)}^{2}}{{d}_{G}}{{(u)}^{2}}+{{(k+1)}^{2}}{{d}_{G}}{{(v)}^{2}}}{{{(k+1)}^{2}}{{d}_{G}}(u){{d}_{G}}(v)}}\\
                    &=&\frac{1}{2(k+1)}\sum\limits_{u\in V(G)}{[{{(k+1)}^{2}}{{d}_{G}}{{(u)}^{2}}+4]}+SSD(G)\\
                   &=&\frac{1}{2(k+1)}[{{(k+1)}^{2}}{{M}_{1}}(G)+4n]+SSD(G).
    \end{eqnarray*}
from where the desired result follows.
\end{proof}

\begin{cor}
For $k\ge 1$ we can obtain the following recurrence relation using theorem 2.8-2.14, from direct calculation.\\
(i) ${{M}_{1}}({{R}_{k}}(G))={{M}_{1}}({{R}_{k-1}}(G))+(2k+1){{M}_{1}}(G)+4m,$\\
(ii) ${{M}_{2}}({{R}_{k}}(G))={{M}_{2}}({{R}_{k-1}}(G))+4k{{M}_{1}}(G)+(2k+1){{M}_{2}}(G),$\\
(iii) $F({{R}_{k}}(G))=F({{R}_{k-1}}(G))+(3k(k+1)+1)F(G)+8m,$\\
(iv) ${{\Pi }_{1}}({{R}_{k}}(G))={{4}^{m}}{{(1+\frac{1}{k})}^{2n}}{{\Pi }_{1}}({{R}_{k-1}}(G)),$\\
(v) ${{\Pi }_{2}}({{R}_{k}}(G))={{k}^{m-n}}{{(k+1)}^{n(k+1)}}{{\Pi }_{2}}(G){{\Pi }_{2}}({{R}_{k-1}}(G)),$
 \begin{eqnarray*}
(vi) HM({{R}_{k}}(G))&=&HM({{R}_{k-1}}(G))+(2k-1)HM(G)+k(3k+1)F(G)+8k{{M}_{1}}(G)\\
                &&+8m,
                \end{eqnarray*}
(vii) $SSD({{R}_{k}}(G))=SSD({{R}_{k-1}}(G))+\frac{1}{2}{{M}_{1}}(G)-\frac{2n}{k(k+1)}.$
\end{cor}

\begin{cor}
Let G be a r-regular graph with $n$ vertices and $m\ (=\frac{nr}{2})$ number of edges. Then, using theorems 2.8-2.14, the following result follows directly:\\
(i) ${{M}_{1}}({{R}_{k}}(G))=n{{r}^{2}}{{(k+1)}^{2}}+2nkr,$\\
(ii) ${{M}_{2}}({{R}_{k}}(G))=2k(k+1)n{{r}^{2}}+\frac{n{{r}^{3}}}{2}{{(k+1)}^{2}},$\\
(iii) $F({{R}_{k}}(G))=n{{r}^{3}}(k+1)+4nkr,$\\
(iv) ${{\Pi }_{1}}({{R}_{k}}(G))={{2}^{knr}}{{r}^{n}}{{(k+1)}^{2n}},$\\
(v) ${{\Pi }_{2}}({{R}_{k}}(G))={{2}^{knr}}{{\{r(k+1)\}}^{nr(k+1)}},$\\
(vi) $HM({{S}_{k}}(G))=4n{{r}^{2}}{{(k+1)}^{2}}+n{{r}^{3}}k{{(k+1)}^{2}}+4n{{r}^{2}}k(k+1)+4nrk,$\\
(vii) $SSD({{R}_{k}}(G))=\frac{1}{2} {n}{{r}^{2}}-\frac{2n}{k(k+1)}+2.$
\end{cor}

Note that, from the results obtained in Theorem 2.8-2.14, we can also obtain the results for ordinary semi total point graph by putting $k=1$ (see \cite{gut15,basa15,de18,basa16}).

\section{conclusion}

In this article, we determine the topological indices such as Zagreb indices, multiplicative Zagreb indices, F-index, hyper-Zagreb index and Symmetric division deg index of k-th subdivision and semi total point graphs respectively. In this paper, we generalized the results for subdivision and semi total point graphs for some $k\ge 1$.

\end{document}